\documentclass{amsart}
\usepackage{amsmath,amsthm,amscd,amsfonts,amssymb,graphicx,color,MnSymbol,enumerate,tikz-cd}
\usepackage{hyperref,cleveref}
\usepackage{xcolor}
\usepackage[labelformat=empty]{caption,subcaption}
\usepackage{cite}
\newtheorem{theorem}{Theorem}[section]
\newtheorem{lemma}[theorem]{Lemma}
\newtheorem{proposition}[theorem]{Proposition}
\newtheorem{corollary}[theorem]{Corollary}
\theoremstyle{definition}
\newtheorem{definition}[theorem]{Definition}

\theoremstyle{remark}
\newtheorem{remark}[theorem]{Remark}

\numberwithin{equation}{section}

\allowdisplaybreaks

\begin{document}
	%------------------------------------------------------------------------------
	
	\title[Abstract Key Polynomials and   MacLane-Vaqui\'e chains]{Abstract Key Polynomials and   MacLane-Vaqui\'e chains}
	\author[Sneha Mavi]{Sneha Mavi}
	\address{Department of Mathematics\\ University of Delhi\\  Delhi-110007, India.}
	\email{mavisneha@gmail.com}
	\author[Anuj Bishnoi]{Anuj Bishnoi$^\ast$}
	\address{Department of Mathematics\\  University of Delhi \\   Delhi-110007, India.}
	\email{abishnoi@maths.du.ac.in}
	
	\begin{abstract}
	In this paper, 	for a  valued field $(K,v)$ of arbitrary rank  and an extension $w$ of $v$ to $K(X),$   a relation between  induced complete sequences of abstract key polynomials and   MacLane-Vaqui\'e chains is given.
	\end{abstract}
	\thanks{$^\ast$Corresponding author, E-mail address: abishnoi@maths.du.ac.in}
	\maketitle
	%------------------------------------------------------------------------------
	
	\section{Introduction }
Let $(K,v)$ be a valued field.   Starting with a valuation $w_0$ of $K(X),$ extending $v,$ which admit key polynomials of degree one,  Nart \cite{EN1} introduced the notion of MacLane-Vaqui\'e  chains
	\begin{align*}
		w_0\xrightarrow{\phi_1,\gamma_1} w_1\xrightarrow{\phi_2,\gamma_2}\cdots \longrightarrow w_{n-1}\xrightarrow{\phi_n,\gamma_n} w_n\longrightarrow\cdots
	\end{align*}
	 consisting of a mixture of ordinary and limit augmentations satisfying some conditions (see Definition \ref{1.1.13}).
	  The main result (Theorem 4.3 of \cite{EN1}) says that all extensions  $w$  of $v$ to $K(X)$ fall exactly in one of the following categories:
	\begin{enumerate}[(i)]
		\item It is the last valuation of a complete finite MacLane-Vaqui\'e chain, i.e., after a finite number $ r$ of augmentation steps,  $w_r=w.$
		\item After a finite number $r$ of augmentation steps, it is the stable limit of a continuous family of augmentations of $w_r$ defined by key polynomials of constant degree.
		\item  It is the stable limit of a complete infinite MacLane-Vaqui\'e chain.
	\end{enumerate}
It is known that   \cite[Theorem 1.1]{NS}, every valuation $w$ on $K(X)$ admits a complete sequence of abstract key polynomials. 
Moreover, in  \cite{MMS} Mahboub et al.\ described a complete sequence of abstract key polynomials  for $w$  satisfying certain properties (see Remark \ref{1.1.3}) and we  call this sequence an \emph{induced complete sequence of abstract key polynomials} for $w.$ In this paper, we prove that the concepts of  MacLane-Vaqui\'e chains and  induced complete sequences of abstract key polynomials for  $w,$ are intimately connected.

	To state the  main result of  the paper,  we first recall some notation, definitions, and preliminary results.

	\section{Notation, Definitions, and Statements of Main Results}
	
	Throughout the paper, $(K,v)$ denote a    valued field of arbitrary rank  with value group $\Gamma_ v$
	 and    residue field  $k_{v},$  and by  $\bar{v}$ we denote an extension of $v$ to a fixed algebraic closure $\overline{K}$ of $K.$

	\medskip

		An extension $w$ of $v$ to the simple transcendental extension $K(X)$ of $K$  such that $k_w$ is algebraic over $k_v$  is said to be \emph{valuation-algebraic} if the quotient $\Gamma_{w}/\Gamma_v$ is a torsion group and   is said to be \emph{value-transcendental}  if $\Gamma_w/\Gamma_v$ is a torsion-free group.  The extension $w$ of $v$ to $K(X)$ is called \emph{residually transcendental}  if 
		the corresponding residue field extension $k_{w}| k_v$ is transcendental.
		We call $w$  \emph{valuation-transcendental} if $w$ is either value-transcendental or is residually transcendental.

	An extension $\overline{w}$ of $w$ to $\overline{K}(X)$ which is also an extension of $\bar{v}$ is called a \emph{common extension} of $w$ and $\bar{v}.$

	\subsection{Abstract key polynomials}
	\begin{definition}
	Let $w$ be a valuation of $K(X)$ and $\overline{w}$ a fixed common extension of $w$ and $\bar{v}$ to $\overline{K}(X).$
For any polynomial $f$ in $K[X],$  we define 
$$\delta(f):=\max\{\overline{w}(X-\alpha)\mid f(\alpha)=0\}.$$
\end{definition}
	This  value $\delta(f)$ does not depend upon the choice of $\overline{w}$ (see  \cite[Proposition 3.1]{JN}).
	\begin{definition}[\bf Abstract key polynomials]
		A monic polynomial $Q$ in $K[X]$ is said to be an \emph{abstract key polynomial}  (abbreviated as ABKP) for $w$ if for each polynomial $f$ in $K[X]$   with $\deg f< \deg Q$ we have $\delta(f)<\delta(Q).$
	\end{definition}
	It is immediate from the definition that all monic  linear polynomials are ABKPs for $w.$  Also an ABKP for $w$ is an irreducible polynomial (see    \cite[Proposition 2.4]{NS}).
	\begin{definition}
		For a polynomial $Q$ in $K[X]$ the \emph{$Q$-truncation} of $w$ is a map $w_Q:K[X]\longrightarrow \Gamma_w$ defined by 
		$$ w_Q(f):= \min_{i\geq 0}\{w(f_iQ^i)\},$$
		where $f=\sum_{i\geq o} f_i Q^i,$ $\deg f_i <\deg Q,$  is the $Q $-expansion of $f.$ 
	\end{definition}
	The $Q$-truncation  $w_Q$  of $w$ need not be a valuation  \cite[Example 2.5]{NS}. However,  if $Q$ is an ABKP for $w,$ then $w_Q$ is a valuation on $K(X)$ (see  \cite[Proposition 2.6]{NS}). Note that any ABKP, $Q$ for $w,$ is also an ABKP for the truncation valuation $w_Q.$

	\begin{definition}\label{1.1.12}
		A family $\Lambda=\{Q_i\}_{i\in\Delta}$ of ABKPs for $w$ is said to be a complete sequence of ABKPs for $w$ if the following conditions are satisfied:
		\begin{enumerate}[(i)]
			\item $\delta(Q_i)\neq \delta(Q_j)$ for every $i\neq j\in\Delta.$
			\item $\Lambda$ is well-ordered with respect to the ordering given by $Q_i< Q_j$ if and only if  $\delta(Q_i)<\delta(Q_j)$ for every $i<j\in \Delta.$ 
			\item For any $f\in K[X],$ there exists some $Q_i \in \Lambda$ such that $\deg Q_i\leq \deg f$ and  $w_{Q_i}(f)=w(f).$
		\end{enumerate}
	\end{definition}
	It is known that  \cite[Theorem 1.1]{NS}, every valuation $w$ on $K(X)$ admits a complete sequence of ABKPs. 
	Moreover, there is a complete sequence  $\Lambda=\{Q_i\}_{i\in\Delta}$ of ABKPs  for  $w$ having the following properties (see   \cite[Remark 4.6]{MMS} and  proof of  \cite[Theorem 1.1]{NS}).
	\begin{remark} \label{1.1.3}
		\begin{enumerate}[(i)]
		\item $\Delta=\bigcup_{j\in I}\Delta_j$ with $I=\{0,1,\ldots, N\}$ or $\mathbb{N}\cup\{0\},$ and for each $j\in I$ we have $\Delta_j=\{j\}\cup\vartheta_{j},$ where $\vartheta_j$ is an ordered set without a last element or is empty.
		\item $\deg Q_0=1.$
		\item For all $j\in I\setminus \{0\}$ we have $j-1<i<j,$ for all $i \in\vartheta_{j-1}.$
		\item All  polynomials $Q_i$ with $i\in\Delta_j$ have the same degree and  have degree strictly  less than  the degree of the polynomials $Q_{i'}$ for every $i'\in\Delta_{j+1}.$
		\item For each $i<i'\in\Delta$ we have $w(Q_i)<w(Q_{i'})$ and $\delta(Q_i)<\delta(Q_{i'}).$
		\end{enumerate}
	\end{remark}
The complete sequences of ABKPs satisfying the properties of Remark \ref{1.1.3} will be called  {\bf induced complete sequences of ABKPs} for $w.$

		Even though the set $\{Q_i\}_{i\in\Delta}$  of ABKPs  for $w$  is not unique, the cardinality of $I$ and the degree of an abstract key polynomial $Q_i$ for each $i\in I$ are uniquely determined by $w.$
	
	The ordered set $\Delta$ has a last element if and only if the following holds:
	\begin{align}\label{1.6}
		I=\{0,1,\ldots,N\}~\text{is finite, and}~ \Delta_N=\{N\}~ \text{(i.e., $\vartheta_N=\emptyset$).}
	\end{align}
	  
\subsection{ MacLane-Vaqui\'e chains}
We first recall the notion of key polynomials which was first introduced by MacLane in 1936 and later generalized by Vaqui\'e in 2007  (see \cite{M, V}).  
\begin{definition}
	For a valuation $w$ on $K(X)$ and polynomials $f,$ $g$ in $K[X],$ we say that
	\begin{enumerate}[(i)]
		\item  $f$ and $g$ are $w$-equivalent and write $f\thicksim_{w} g$ if $w(f-g)>w(f)=w(g).$
		\item  $f$  $w$-divides $g$ (denoted by $f\mid_{w}g$) if there exists some polynomial $h \in K[X]$ such that $g\thicksim_{w} fh.$
		\item  $f$ is $w$-irreducible if for any $h,\, q\in K[X],$ whenever $f\mid_{w} hq,$ then either $f\mid_{w}h $ or $f\mid_{w}q.$
		\item $f$ is $w$-minimal if for every nonzero polynomial $h\in K[X],$ whenever $f\mid_{w}h,$ then $\deg h\geq \deg f.$
	\end{enumerate}
	\end{definition}
 \begin{definition}[\bf Key polynomials]
 	A monic polynomial $f$ in $K[X]$   is called a  \emph{key polynomial} for $w,$ if $f$ is $w$-irreducible and $w$-minimal.
 \end{definition}
 In view of  Proposition 2.10 of  \cite{Ma}  any ABKP, $Q$ for $w$ is a key polynomial for $w_Q$ of minimal degree. 
 Let $\operatorname{KP}(w)$ denote the set of all key polynomials for  $w.$ Then for any $\phi\in \operatorname{KP}(w)$ we denote by $[\phi]_w$ the set of all key polynomials which are $w$-equivalent to $\phi.$ For any  $\phi,$ $\phi'\in \operatorname{KP}(w),$ we have 
 $$\phi\mid_{w}\phi'~\text{if and only if}~ \phi\sim_{w}\phi',$$
 and in this case, $\deg\phi=\deg\phi'$ (see  \cite[Proposition 6.6]{EN}).

 Let  $w$ be a valuation on  $K(X)$  which admits key polynomials.
 If $\phi$ is a key polynomial for $w$ of minimal degree, then we define
   $$\deg(w):=\deg \phi.$$  For any valuation $w'$ on $K(X)$ taking values in a subgroup of $\Gamma_{w},$ we say that $$w'\leq w~\text{ if and only if}~ w'(f)\leq w(f),~\forall~f\in K[X].$$   If $w'<w,$ we denote by $\Phi(w',w)$
  the set of all monic polynomials $g\in K[X]$ of minimal degree (say) $d$ such that $w'(g)<w(g).$   We denote  $$\deg(\Phi(w',w))=d.$$
 
 \begin{definition}[\bf Ordinary augmentation]\label{1.1.15}
 	Let $\phi$ be a key polynomial for a valuation $w'$ on $K(X)$ and $\gamma> w'(\phi)$ be an element of a totally ordered abelian group $\Gamma$ containing $\Gamma_{w'}$ as an ordered subgroup. The map $w: K[X]\longrightarrow \Gamma\cup\{\infty\}$ defined by 
 	$$w(f):=\min_{i\geq 0}\{w'(f_i)+i\gamma\},$$ where
 	$f=\sum_{i\geq 0}f_i \phi^i, $ $\deg f_i<\deg \phi,$ is the $\phi$-expansion of $f\in K[X],$  gives a valuation on $K(X)$ (see  \cite[Theorem 4.2]{M}) called the   \emph{ordinary augmentation of $w'$},  and is  denoted  by $w=[w'; \phi,\gamma].$
 \end{definition}
 Note that $w(\phi)=\gamma,$ i.e., $w'(\phi)<w(\phi)$ and
 the polynomial $\phi$ is a key polynomial of minimal degree for the augmented valuation $w$ (see  \cite[Corollary 7.3]{EN}).
 \begin{theorem}[Theorem 1.15, \cite{V}]\label{1.1.19}
 	Let $w$ be a valuation on $K(X)$ and $w'<w.$ Then 
  any $\phi\in \Phi(w',w)$ is a key polynomial for $w'$ and  $$w'<[w';\phi,w(\phi)]\leq w.$$  For any nonzero polynomial $f\in K[X],$ the equality 
 $w'(f)=w(f)$ holds if and only if $\phi\nmid_{w'} f.$
\end{theorem}

 \begin{corollary}[Corollary 2.5, \cite{EN1}]\label{1.1.16}
 	Let $w'<w$ be as above. Then
 	\begin{enumerate}[(i)]
 		\item  $\Phi(w',w )=[\phi]_{w'}$ for all $\phi\in\Phi(w',w ).$
 		\item If $w'<\nu\leq w$ is a chain  of valuations, then $ \Phi(w',w )=\Phi(w',\nu).$
 		In particular, 
 		\begin{align*}
 			w'(f)=w(f) \iff w'(f)=\nu(f),\hspace{2pt}\forall f\in K[X]. 
 		\end{align*}
 	\end{enumerate}
 \end{corollary}
 \begin{corollary}\label{1.1.17}
 	If $w=[w';\phi,\gamma]$ is an ordinary augmentation, then $\Phi(w',w )=[\phi]_{w'}.$
 \end{corollary}

Consider a totally ordered family of valuations  on $K(X),$ taking  values in a common ordered group $$\mathcal{W}=(\rho_i)_{i\in\mathbf{A}},$$  and  indexed by a totally ordered set $\mathbf{A}.$ We shall always assume that the assignment 
$i\mapsto\rho_i$ is an isomorphism between  totally ordered sets  $\mathbf{A}$ and $\mathcal{W}.$

A polynomial $f$ in $K[X]$ is said to be \emph{$\mathcal{W}$-stable}  if 
$$\rho_i(f)=\rho_{i_0}(f),~ \forall~ i\geq i_0,$$ for some index $i_0\in\mathbf{A}.$ This stable value is denoted by $\rho_\mathcal{W}(f).$
We obtain in this way a \emph{stability function} $\rho_{\mathcal{W}}$ defined only on the set of stable polynomials which is a multiplicatively closed subset of $K[X].$ 

 In view of Corollary \ref{1.1.16} (ii), a polynomial $f\in K[X]$  is \emph{$\mathcal{W}$-unstable} if and only if  
$$\rho_i(f)<\rho_j(f),\hspace{5pt} \forall~ i<j\in\mathbf{A}.$$  
We denote 
 $$m_{\infty}=\min\{\deg f\mid f\in K[X],~\text{$f$ is $\mathcal{W}$-unstable}\}.$$ If  all polynomials are $\mathcal{W}$-stable, then we set $m_{\infty}=\infty$. 
	We say that $\mathcal{W}$ has a \emph{stable limit} if all polynomials in $K[X]$ are $\mathcal{W}$-stable. In this case, $\rho_{\mathcal{W}}$ is a valuation on $K[X],$ and is  called the \emph{stable limit} of $\mathcal{W}.$

\begin{definition}\label{1.1.14}
	Let $w'$ be a valuation on $K(X)$ admitting key polynomials. Then a \emph{continuous family of augmentations} of $w'$ is a family of ordinary augmentations of $w'$  $$\mathcal{W} =(\rho_i=[w';\chi_i,\gamma_i])_{i\in	\mathbf{A}},$$ indexed by a totally ordered set $\mathbf{A}$ such that $\gamma_i<\gamma_j$ for all $i<j$ in $\mathbf{A},$ satisfying the following conditions:
	\begin{enumerate}[(i)]
		\item The set $\mathbf{A}$  has no last element.
		\item All  key polynomials $\chi_i\in \operatorname{KP}(w')$ have the same degree.
		\item For all $i<j$ in $\mathbf{A},$ $\chi_j$ is a  key polynomial for $\rho_i,$ 
		$\chi_j\not\thicksim_{\rho_i}\chi_i~\text{and}~ \rho_j=[\rho_i;\chi_j,\gamma_j].$
	\end{enumerate}
\end{definition}
	The common degree $\deg\chi_i,$ for all $i,$ is called the \emph{stable degree} of the family $\mathcal{W}$ and is denoted by $\deg (\mathcal{W}).$
\begin{remark}\label{1.1.18}
The following properties hold for any  continuous family $\mathcal{W} =(\rho_i)_{i\in\mathbf{A}}$ of augmentations (see p.\ 9, \cite{EN1}):
\begin{enumerate}[(i)]
	\item The mapping defined by $i\mapsto\gamma_i$ and $i\mapsto\rho_i$ are isomorphisms of ordered sets between $\mathbf{A}$ and $\{\gamma_i\mid i\in\mathbf{A}\},$ $\{\rho_i\mid i\in\mathbf{A}\},$ respectively.
	\item For all $i\in\mathbf{A},$ $\chi_i$ is a key polynomial for $\rho_i$ of minimal degree.
	
	\item For all $i, j\in\mathbf{A},$ $\rho_i(\chi_j)=\min\{\gamma_i,\gamma_j\}.$ Hence, all the polynomials $\chi_i$ are stable.
	\item $\Phi(\rho_i,\rho_j)=[\chi_j]_{\rho_i},$ $\forall$ $i<j\in\mathbf{A}.$
	\item All valuations $\rho_i$ are residually transcendental.
	\item All the value groups $\Gamma_{\rho_i}$ coincide and   the common value group is denoted by $\Gamma_{\mathcal{W}}.$ 
\end{enumerate}
\end{remark}
\begin{remark}
 Since a totally ordered set admits a well-ordered cofinal subset, so without loss of generality  we can assume that $\mathbf{A}$ is  well-ordered.
\end{remark}
\begin{definition}[\bf MacLane-Vaqui\'e limit key polynomials]
	Let $\mathcal{W}$ be a continuous family of augmentations of a valuation $w'.$ Then
 a   monic $\mathcal{W}$-unstable polynomial of minimal degree  is called a \emph{MacLane-Vaqui\'e  limit key polynomial} (abbreviated as MLV) for $\mathcal{W}.$
\end{definition}
  The set of all MLV limit key polynomials is denoted by $\operatorname{KP}_{\infty}(\mathcal{W}).$  Since the product of stable polynomials is stable, so all MLV limit key polynomials are irreducible in $K[X].$ 

 	Any continuous family $\mathcal{W}$ of augmentations of $w'$ fall in one of the following three cases:
 	\begin{enumerate}[(i)]
 		\item  It has a stable limit,  i.e., $\rho_{\mathcal{W}}$ is a valuation on $K[X],$ if $m_{\infty}=\infty.$ 
 		\item It is \emph{in-essential} if $m_{\infty}=\deg(\mathcal{W})$ (stable degree).
 		\item It is  \emph{essential }  if $\deg(\mathcal{W})<m_{\infty}<\infty.$
 	\end{enumerate}

Let $\mathcal{W}$ be an essential continuous family of augmentations of a valuation $w'.$ Then 	$\mathcal{W}$ admit MLV limit key polynomials.
If $Q$ is an MLV limit key polynomial, then any polynomial $f$ in $K[X]$ with $\deg f<\deg Q$ is  $\mathcal{W}$-stable.
\begin{definition}[\bf Limit augmentation]
	Let   $Q$ be any MLV limit key polynomial for an essential continuous family $\mathcal{W}  =(\rho_i)_{i\in\mathbf{A}}$ of augmentations of $w'$   and   $\gamma>\rho_i(Q),$ for all $i\in\mathbf{A},$ be an element of  a totally ordered abelian group $\Gamma$  containing $\Gamma_{\mathcal{W}}$ as an  ordered subgroup. Then the map $w: K[X]\longrightarrow\Gamma\cup\{\infty\}$ defined by
	$$w(f):=\min_{i\geq 0}\{\rho_\mathcal{W}(f_i)+i\gamma\},$$
where  $f=\sum_{i\geq 0} f_iQ^i,$ $\deg f_i<\deg Q,$ is the $Q$-expansion of $f\in K[X],$ gives a valuation on $K(X)$ and is called  the \emph{limit augmentation} of $\mathcal{W},$  denoted by $w=[\mathcal{W}=(\rho_i)_{i\in\mathbf{A}}; Q, \gamma].$ 
\end{definition}
 Note that $w(Q)=\gamma$ and $\rho_i<w$ for all $i\in\mathbf{A}.$ Also,  $Q$ is a key polynomial for $w$ of minimal degree  \cite[Corollary 7.13]{EN}.
\vspace{.20pt}

We now recall the definition of MacLane-Vaqui\'e chains given by  Nart in \cite{EN1}. For this,  we first
consider a finite, or countably infinite, chain of mixed augmentations
\begin{align}\label{1.12}
	w_0\xrightarrow{\phi_1,\gamma_1} w_1\xrightarrow{\phi_2,\gamma_2}\cdots \longrightarrow w_{n}\xrightarrow{\phi_{n+1},\gamma_{n+1}} w_{n+1}\longrightarrow\cdots
\end{align}
in which every valuation is an augmentation of the previous one and is of one of the following type:
\begin{itemize}
	\item  Ordinary augmentation: $w_{n+1}=[w_n; \phi_{n+1},\gamma_{n+1}],$ for some $\phi_{n+1}\in \operatorname{KP}(w_n).$ 
	\item Limit augmentation:  $w_{n+1}=[\mathcal{W}_n; \phi_{n+1},\gamma_{n+1}],$ for some $\phi_{n+1}\in \operatorname{KP}_{\infty}(\mathcal{W}_n),$ where $\mathcal{W}_n$ is an essential continuous family of augmentations of $w_n.$ 
\end{itemize}
Let $\phi_0\in \operatorname{KP}(w_0)$ be a key polynomial of minimal degree and let $\gamma_0=w_0(\phi_0).$
 Then, in view of  Theorem \ref{1.1.19},  Proposition 6.3 of \cite{EN}, Proposition 2.1, 3.5 of \cite{EN1} and Corollary \ref{1.1.16},   a  chain (\ref{1.12}) of augmentations have the following properties:
 \begin{remark}\label{1.1.20}
\begin{enumerate}[(i)]
	\item  $\gamma_n=w_n(\phi_n)<\gamma_{n+1}.$
	\item For all $n\geq 0,$  the polynomial $\phi_n$ is a key polynomial for $w_n$ of minimal degree and therefore
	$$\deg(w_n)(=\deg \phi_n)~\text{divides}~\deg (\Phi(w_n,w_{n+1})).$$
	\item \begin{equation*}
		\Phi(w_n,w_{n+1})=
		\begin{cases}
			[\phi_{n+1}]_{w_n},&\text{if  $w_n\rightarrow w_{n+1}$ is an ordinary augmentation}\\
		\displaystyle\Phi_{w_n,\mathcal{W}_n}=[\chi_i]_{w_n},&\text{if  $w_n\rightarrow w_{n+1}$ is a limit augmentation}.
		\end{cases}
	\end{equation*}
	\item \begin{equation*}
			\deg(\Phi(w_n,w_{n+1}))=
			\begin{cases}
				\deg\phi_{n+1},&\text{if}~ w_n\rightarrow w_{n+1}~ \text{is an ordinary augmentation}\\
				\displaystyle\deg(\mathcal{W}_n),&\text{if}~ w_n\rightarrow w_{n+1} ~\text{is a limit augmentation}.
			\end{cases}
	\end{equation*}
\end{enumerate}
\end{remark}
A valuation $w$ on $K(X)$
  is called a \emph{depth zero valuation} if 
  $w=w_{\alpha,\delta},$ for some $ \alpha \in K$ and  $\delta\in\Gamma,$  where $w_{\alpha,\delta}$ is a valuation on $K[X]$  defined by
  $$w_{\alpha,\delta}\left(\sum_{i\geq 0} c_i (X-\alpha)^i\right):=\min_{i\geq 0}\{v(c_i)+i\delta\}, \, c_i\in K.$$
\begin{definition}[\bf MacLane-Vaqui\'e chains]\label{1.1.13}
	A finite, or countably infinite chain of mixed augmentations as in (\ref{1.12}) is called a \emph{MacLane-Vaqui\'e  chain} (abbreviated as MLV chain), if every augmentation step satisfies:
	\begin{enumerate}[(i)]
		\item if $w_n\rightarrow w_{n+1}$ is an ordinary augmentation, then $\deg (w_n)<\deg(\Phi(w_n,w_{n+1})).$
		\item if $w_n\rightarrow w_{n+1}$ is a limit augmentation, then $\deg (w_n)=\deg(\Phi(w_n,w_{n+1}))$ and $\phi_n\notin\Phi(w_n,w_{n+1}).$
	\end{enumerate}
	A   MacLane-Vaqui\'e  chain is said to be \emph{complete} if  $w_0$ is a  depth zero valuation.
\end{definition}
\begin{remark}\label{1.1.10}
	Every infinite MLV chain (\ref{1.12}) has a stable limit.  Since  for any polynomial $f\in K[X],$  there exist some $n\geq 0$ such that $\deg f<\deg \phi_n$ and as $\deg\phi_n=\deg(w_n)\leq \deg (\Phi(w_n, w_{n+1})),$ so by Theorem \ref{1.1.19},  $w_n(f)=w_{n+1}(f).$  Therefore $\rho_{\mathcal{W}}$ is the stable limit of the ordered family of valuations $\mathcal{W}=(w_n)_{n\in I},$ where $I=\mathbb{N}\cup\{0\}.$ 
\end{remark}

In Theorem  3.1 of \cite{SA},  given an induced complete sequence of ABKPs, $\{Q_i\}_{i\in\Delta}$  for   $w$ such that $\Delta$ has a last element a precise complete finite MLV chain of $w$ is obtained, and conversely if $w$ has a  complete finite MLV chain, then Theorem 3.2 of \cite{SA} gives a construction of an induced complete sequence of ABKPs of the above type. Suppose now that $w$ has an induced complete sequence of ABKPs such that $\Delta$ has no last element. In the following result, using such a complete sequence we give an explicit construction of an MLV chain of $w.$ 
\begin{theorem}\label{1.1.2}
	Let $(K,v)$ be a valued field and let $w$ be an extension of $v$ to $K(X).$ If $\{Q_i\}_{i\in\Delta}$ is an induced complete sequence of ABKPs for $w$  such that $\Delta$ has no last element, then $w$ falls in exactly one of the following two cases. 
	\begin{enumerate}[(i)]
		\item After a finite number, say, $r$ of augmentation steps, it is the stable limit of a continuous family ${\mathcal{W}_r}$ of augmentations of $w_r$
		\begin{align*}
			w_0\xrightarrow{Q_1,\gamma_1} w_1\xrightarrow{Q_2,\gamma_2}\cdots \longrightarrow w_{r-1}\xrightarrow{Q_r,\gamma_r} w_r\xrightarrow {{\mathcal{W}_r}} \rho_{\mathcal{W}_r}=w,
		\end{align*}
	such that $\deg(\Phi(w_r,w))=\deg(w_r)$ and $Q_r\notin \Phi(w_r, w).$
	\item It is the stable limit of a complete infinite MLV chain. 
	\begin{align*}
		w_0\xrightarrow{Q_1,\gamma_1} w_1\xrightarrow{Q_2,\gamma_2}\cdots \longrightarrow w_{n}\xrightarrow{Q_{n+1},\gamma_{n+1}} w_{n+1}\longrightarrow\cdots
	\end{align*}
	\end{enumerate}
 In both cases, $\gamma_j=w(Q_j)$ for all $j\in I.$ Also, an augmentation $w_j\longrightarrow w_{j+1}$ is ordinary if and only if $\vartheta_{j}=\emptyset.$
\end{theorem}
The converse of the above result also holds.
\begin{theorem}\label{1.1.4}
	Let $(K,v)$ be a valued field and let $w$ be an extension of $v$ to $K(X)$ such that $w$ falls in exactly one of the following two cases.
	\begin{enumerate}[(i)]
		\item  After a finite number, say, $r$ of augmentation steps, it is the stable limit of a continuous family $\mathcal{W}_r=(\rho_i)_{i\in\mathbf{A}_r}$ of augmentations of $w_r$
		\begin{align}\label{1.1}
			w_0\xrightarrow{\phi_1,\gamma_1} w_1\xrightarrow{\phi_2,\gamma_2}\cdots \longrightarrow w_{r-1}\xrightarrow{\phi_r,\gamma_r} w_r\xrightarrow {(\rho_i)_{i\in\mathbf{A}_r}} \rho_{\mathcal{W}_r}=w
		\end{align}
		such that $\deg(\Phi(w_r,w))=\deg(w_r)$ and $\phi_r\notin \Phi(w_r, w).$\\
		
		\item  It is the stable limit of a complete infinite MLV chain, 
		\begin{align}\label{1.2}
			w_0\xrightarrow{\phi_1,\gamma_1} w_1\xrightarrow{\phi_2,\gamma_2}\cdots \longrightarrow w_{n}\xrightarrow{\phi_{n+1},\gamma_{n+1}} w_{n+1}\longrightarrow\cdots
		\end{align}
	\end{enumerate}
	 Then, there is a totally ordered set $\Delta$ containing no last element and an induced complete sequence
	  $\Lambda=\{Q_i\}_{i\in\Delta}$ of ABKPs for $w,$ constructed as follows:
	    \begin{enumerate}[(a)]
	    	\item $\Delta=\bigcup_{j\in I}\Delta_{j},$ with $I=\{0,1,\ldots,r\}$ in case (i) and $I=\mathbb{N}\cup\{0\}$ in case (ii).
	    	\item $\Delta_{j}=\{j\}\cup\vartheta_{j},$ for all $j\in I,$ and $\vartheta_{j}=\emptyset$ if and only if the augmentation $w_j\longrightarrow w_{j+1}$ is ordinary. Moreover, $Q_j=\phi_j$ for all $j\in I.$
	    	\item Suppose that $w_j\longrightarrow w_{j+1}$ is a limit augmentation, or (in case (i)) $j=r$ and $w_j\longrightarrow w$ is a stable limit step. Let $\mathcal{W}_j=(\rho_{i})_{i\in\mathbf{A}_j}$ be the underlying  totally ordered family. Then $\vartheta_{j}=\mathbf{A}_j$ and $Q_i=\chi_i$ for all $i\in\mathbf{A}_j,$ where  $\chi_i$ is the key polynomial for $w_j$ such that $\rho_i=[w_j;\chi_i, w(\chi_i)].$
	    \end{enumerate}
\end{theorem}
It is known that if  $\{Q_i\}_{i\in\Delta}$ is a complete sequence of ABKPs for $w,$ then $w$ is a valuation-transcendental extension of $v$ to $K(X)$ if and only if $\Delta$ has a last element, say, N, and then 	 $w=w_{Q_N}$ (see  \cite[Theorem 5.6]{MMS}). Therefore, as an immediate consequence of Theorems \ref{1.1.2} and \ref{1.1.4} we have the following result.
\begin{corollary}\label{1.1.9}
	Let $(K,v)$ and $(K(X),w)$ be as above. Then the following are equivalent:
	\begin{enumerate}[(i)]
		\item The extension $w$ is valuation-algebraic.
		\item There exists an induced complete sequence  $\{Q_i\}_{i\in\Delta}$ of ABKPs for $w$ such that $\Delta$ has no last element.
		\item  The extension $w$ has an MLV chain of type (\ref{1.1}) or (\ref{1.2}).
	\end{enumerate}
\end{corollary}

	\section{Preliminaries}
	Let $(K,v)$ be a valued field and  $(\overline{K},\bar{v})$ be as before. Let $w$ be an extension of $v$ to $K(X)$  and $\overline{w}$  a common extension of $w$ and $\bar{v}$ to $\overline{K}(X).$ In this section we give some preliminary results which will be used to prove the main results. 
	
	We first recall some basic properties of ABKPs for  $w$   (see Lemma 2.11 of \cite{NS} and  Proposition 3.8,  Corollary  3.11,  Theorem 6.1 of  \cite{JN1}).
	\begin{proposition}\label{2.1.6}
		For  ABKPs, $Q$ and $Q'$ for $w$ the following holds:
		\begin{enumerate}[(i)]
				%\item If $w_Q<w,$ then $w_Q$ is an r.\ t.\ extension.
				\item If $\delta(Q)<\delta(Q'),$ then $w_{Q}(Q')<w(Q').$
			\item Suppose that  $\delta(Q)<\delta(Q').$ For any polynomial $f\in K[X],$ we have 
			\begin{align*}
				w_{Q}(f)=w(f)&\implies w_{Q'}(f)=w(f),\\
				w_{Q'}(f)<w(f)&\implies w_Q(f)<w_{Q'}(f).
			\end{align*}	 
			\item If $Q'\in \Phi(w_Q,w),$ then $Q$ and $Q'$ are key polynomials for $w_Q.$ Moreover, $w_{Q'}=[w_Q; Q', w(Q')].$
			\item 	 Every $F\in\Phi(w_Q,w)$ is an ABKP for $w$ and $\delta(Q)<\delta(F).$  
		\end{enumerate}
	\end{proposition}

	The next two results gives a comparison between key polynomials and ABKPs.
	\begin{theorem}[Theorem 2.17, \cite{Ma}]\label{2.1.2}
		Suppose that $w'<w$ be valuations on $K(X)$ and $\phi$  a key polynomial for $w'.$ Then $\phi$ is an ABKP polynomial for $w$ if and only if it satisfies one of the following two conditions:
		\begin{enumerate}[(i)]
			\item $\phi\in \Phi(w',w).$
			\item $\phi\notin \Phi(w',w)$ and $\deg \phi=\deg (w').$
		\end{enumerate}
		In the first case $w_{\phi}=[w'; \phi, w(\phi)].$ In the second case $w_{\phi}=w'.$
	\end{theorem}
\begin{theorem}[Theorem 2.18, \cite{Ma}]\label{2.1.7}
	Let $\phi\in \operatorname{KP}(w).$ Then $\phi$ is an ABKP for $w$ if and only if $\deg\phi=\deg(w).$ In this case, $w_\phi=w.$
\end{theorem}

 As an application of the above two theorems, we have  the following two results.
\begin{lemma}\label{1.1.11}
	 	 Let $w=[w';\phi,\gamma]$ be an ordinary augmentation of a valuation   $w'$ of $K(X).$ If $\phi'$ is a minimal degree key polynomial for $w',$  then both $\phi'$ and $\phi$ are ABKPs for $w.$ Moreover, if  $\phi\nmid_{w'}\phi',$ then  $\delta(\phi')<\delta(\phi).$
	\end{lemma}
\begin{proof}
	Since  $\phi'$ is a minimal degree key polynomial for $w',$   Theorem \ref{2.1.2} shows that $\phi'$ is an ABKP for $w$ and as $\phi$ is a minimal degree key polynomial for $w,$ so by Theorem \ref{2.1.7}  $\phi$ is an ABKP for $w.$

Now, suppose that  $\phi\nmid_{w'}\phi'.$ By Corollary \ref{1.1.17}, $\phi'\notin\Phi(w',w)=[\phi]_{w'}.$	Assume first that $\deg\phi'=\deg\phi.$ Then by Theorem \ref{2.1.2}, we have $w'=w_{\phi '}.$
	  Hence, Proposition  \ref{2.1.6} (iv) shows that $\delta(\phi')<\delta(\phi).$ 
	 Finally, if $\deg\phi'<\deg\phi,$  then    $\delta(\phi')<\delta(\phi)$  because $\phi$ is an ABKP for $w.$
	\end{proof}

\begin{lemma}\label{2.1.1}
Let $w'<w$ be valuations of $K(X).$ If $\mathcal{W}=(\rho_i=[w';\chi_i,\gamma_i])_{i\in\mathbf{A}}$ is a continuous family of augmentations of $w'$ such that $\rho_i<w$ for all $i\in\mathbf{A},$ then   each  $\chi_i$ is an ABKP for $w,$ $\rho_i=w_{\chi_i}$ and $\delta(\chi_i)<\delta(\chi_j)$ for all $i<j\in\mathbf{A}.$ Moreover, if $\phi'$ is a minimal degree key polynomial for $w',$ then $\phi'$ is also an ABKP for $w,$ and if  $\chi_i\nmid_{w'}\phi',$ then $\delta(\phi')<\delta(\chi_i)$ for all $i\in\mathbf{A}.$
\end{lemma}
\begin{proof}
	Since  $\mathcal{W}=(\rho_i)_{i\in\mathbf{A}}$ is a continuous family of augmentations of $w',$  so for all $i<j\in\mathbf{A},$ we have that $\chi_j\not\sim_{\rho_i}\chi_i$ and  $\rho_j=[\rho_i;\chi_j,\gamma_j]$ is an ordinary augmentation of $\rho_i.$ Therefore, by   Remark \ref{1.1.18} (iv) and Corollary \ref{1.1.16} (ii) we have that 
	\begin{align*}
		\chi_i\notin\Phi(\rho_i,\rho_j)=[\chi_j]_{\rho_i}=\Phi(\rho_i,w),
	\end{align*}
	i.e., $\rho_i(\chi_i)=\rho_j(\chi_i),$ which in view of Corollary \ref{1.1.16} (ii),  implies that  $\rho_i(\chi_i)=w(\chi_i),$ and as $\deg\chi_i=\deg\rho_i,$ 
	 so by Theorem \ref{2.1.2}, $\chi_i$ is an ABKP for $w$ and
	 \begin{align}\label{1.13}
	 	\rho_i=w_{\chi_i}~\text{for all $i\in\mathbf{A}$}.
	 \end{align}
	 Again from Corollary \ref{1.1.16} (ii) and  Remark \ref{1.1.18} (iv), it follows that $$\chi_j\in[\chi_j]_{\rho_i}=\Phi(\rho_i,\rho_j)=\Phi(\rho_i,w)~\text{for all $i<j\in\mathbf{A},$}$$
	 i.e.,
	$\rho_i(\chi_j)<w(\chi_j),$ which on using (\ref{1.13}), together with the fact that
	  $\deg\chi_i=\deg\chi_j,$ implies  that $\chi_j\in\Phi(w_{\chi_i},w).$  Hence from Proposition \ref{2.1.6} (iv), $\chi_j$ 
	is an ABKP for $w$ and 
	\begin{align*}
		\delta(\chi_i)<\delta(\chi_j)~\text{for all $i<j\in\mathbf{A}$}.
	\end{align*}
	Since 
	  $w'<w,$  so $w'(\phi')\leq w(\phi').$
	  Now on using Theorem \ref{2.1.2}, together with the hypothesis that $\phi'$ is a minimal degree key polynomial for $w',$ i.e., $\deg\phi'=\deg (w'),$  we get that $\phi'$ is an ABKP for $w.$
	   Keeping in mind   that $\chi_i\nmid_{w'} \phi',$  it immediately  follows from  Lemma \ref{1.1.11}, that
	  $	\delta(\phi')<\delta(\chi_i)~\text{for all $i\in\mathbf{A}.$}$
\end{proof}

\begin{remark}\label{2.1.9}
	In the above lemma, if $w=[\mathcal{W};\phi,\gamma]$ is a limit augmentation of an essential continuous family $\mathcal{W}=(\rho_i=[w';\chi_i,\gamma_i])_{i\in\mathbf{A}}$  of augmentations of   $w',$  then  by Theorem \ref{2.1.7}, $\phi$ is an ABKP for $w,$ and as  $\phi$ is an MLV limit key polynomial for $\mathcal{W},$ so $\deg\chi_i<\deg\phi.$ Consequently,
	$$\delta(\chi_i)<\delta(\phi), ~\forall~i\in\mathbf{A}.$$
\end{remark}
In the next  result we give some properties of an induced complete sequence of ABKPs.
\begin{lemma}\label{2.1.5}
	Let $\{Q_i\}_{i\in\Delta}$ be an induced complete sequence of ABKPs for $w$ such that $\vartheta_{j}\neq \emptyset$ for some $j\in\Delta.$ Then  the following holds:
	\begin{enumerate}[(i)]
		\item $Q_{i'}\in\Phi(w_{Q_i},w)$ for every $i<i'\in\Delta_{j}=\{j\}\cup\vartheta_{j}.$
		\item $\mathcal{W}_j=(w_{Q_i})_{i\in\vartheta_{j}}$ is a continuous family of augmentations of $w_{Q_j}.$
	\end{enumerate}
		\end{lemma}
\begin{proof}
	\noindent(i)~
	Follows from  Remark \ref{1.1.3} and Proposition \ref{2.1.6} (i).
\smallskip
	
		\noindent (ii)~	Since for each $i\in\Delta,$ $Q_i$ is an ABKP for $w,$ so $w_{Q_i}$ is a valuation on $K(X).$ 
		By hypothesis, as $\vartheta_j\neq \emptyset,$ so by (i),  for each $i\in\vartheta_j$ we have that  $Q_i\in\Phi(w_{Q_j},w)$ and $\deg Q_i=\deg Q_j,$ where $Q_j$ is the ABKP corresponding to  $\{j\}.$  From Proposition \ref{2.1.6} (iii), it follows that each $Q_i$ is a key polynomial for $w_{Q_j}$ and  
		$$w_{Q_i}=[w_{Q_j};Q_i, w(Q_i)].$$ Similarly, for each $i<i'\in\vartheta_{j},$ we get that 
		$Q_{i'}$ is a key polynomial for $w_{Q_i}$ and $$w_{Q_{i'}}=[w_{Q_i}; Q_{i'}, w(Q_{i'})].$$  Now by Corollary \ref{1.1.16} (ii),  $\Phi(w_{Q_i},w)=\Phi(w_{Q_i}, w_{Q_{i'}}),$ and as $Q_i\notin\Phi(w_{Q_i},w),$ so 
		\begin{align*}
			w_{Q_i}(Q_{i})=w_{Q_{i'}}(Q_i),~ \forall ~ i<i'\in\vartheta_{j},
		\end{align*}
		which in view of  Theorem \ref{1.1.19}, implies that $Q_{i'}\not\sim_{w_{Q_i}} Q_i.$
		Hence 
		$\mathcal{W}_{j}=(w_{Q_i})_{i\in\vartheta_j}$ is a continuous family of augmentations of $w_{Q_j}.$
	 \end{proof}

	\section{Proof of Main Results}

	\begin{proof}[Proof of Theorem \ref{1.1.2}]
		Since $\{Q_i\}_{i\in\Delta}$ is an induced complete sequence of ABKPs for $w,$  so by Remark \ref{1.1.3} (i), $\Delta=\bigcup_{j\in I}\Delta_j,$ where $I=\{0,1,\ldots, N\}$ or $I=\mathbb{N}\cup\{0\}$ and for each $j\in I,$  $\Delta_j=\{j\}\cup\vartheta_{j},$  where $\vartheta_j$ is either  empty or an ordered set without a last element. In view of hypothesis as $\Delta$ has no last element, so by  (\ref{1.6}), either $I=\mathbb{N}\cup\{0\}$  or  $I=\{0,1,\ldots, N\}$ and $\Delta_N=\{N\}\cup\vartheta_N,$ with $\vartheta_N\neq \emptyset.$ 
		Since for each $i\in\Delta,$ $Q_i$ is an ABKP for $w,$ so $w_{Q_i}$ is a valuation on $K(X)$ and we  denote it by $w_i.$  
		  
		Arguing as in the proof of  \cite[Theorem 3.1]{SA}, we get that each $Q_i,$ $i\in\Delta$ is a key polynomial for $w_{Q_i}$ of minimal degree and 
		\begin{enumerate}[(1)]
		\item if $\vartheta_{j}= \emptyset,$ then $w_j\longrightarrow w_{j+1}$ is an ordinary augmentation.
		\item if $\vartheta_{j}\neq \emptyset,$ then $w_j\longrightarrow w_{j+1}$ is a limit augmentation of an essential continuous family $\mathcal{W}_j=(\rho_i=w_{Q_i})_{i\in\vartheta_{j}}$ of augmentations of $w_j.$
	\end{enumerate}
	In either case, denote
	\begin{align}\label{1.5}
		\gamma_i=w(Q_i),~\forall~i\in\Delta.
	\end{align}

		Suppose first that $I=\{0,1,\ldots, N\}$  and $\vartheta_N\neq \emptyset.$ Then from above
		\begin{align}\label{1.4}
			w_0\xrightarrow{Q_1,\gamma_1} w_1\xrightarrow{Q_2,\gamma_2}\cdots \longrightarrow w_{N-1}\xrightarrow{Q_N,\gamma_N} w_N<w,
		\end{align}
	   is a finite MLV chain  of $w_N,$  
	  	and  this chain is complete because $w_0=w_{Q_0},$ where $\deg Q_0=1,$   is a depth zero valuation.  Since $\vartheta_N\neq \emptyset,$ so by Lemma \ref{2.1.5} (ii), we have that
	  $\mathcal{W}_N=(\rho_i=w_{Q_i})_{i\in\vartheta_N}$ is a continuous family of augmentations of $w_N.$ We now claim the following:
	  	\begin{enumerate}[(a)]
	  		\item  $w$ is the stable limit  of $\mathcal{W}_N,$
	  		\item $\deg (\Phi(w_N, w))=\deg (w_N)$ and $Q_N\notin \Phi(w_N, w).$
	  	\end{enumerate}
	  	Let $f$ in $K[X]$ be any polynomial.  As $\Lambda$ is complete, so there exists some $j_0\in\Delta$ such that 
	\begin{align}\label{1.3}
			w_{j_0}(f)=w(f).
	\end{align}
Take  $i\in\vartheta_N$ such that  $j_0<i.$ Then,   $\delta(Q_{j_0})<\delta(Q_i)$ and,  in view of  (\ref{1.3}),  Proposition \ref{2.1.6} (ii) shows that $w_{i}(f)=w(f).$ Therefore, $$\rho_{\mathcal{W}_N}(f)=w_{j_0}(f)=w_{i}(f)=w(f)~\text{for every $i\in\vartheta_N$}.$$
 Thus, every polynomial in $K[X]$ is $\mathcal{W}_N$-stable, i.e., $\rho_{\mathcal{W}_N}$ is a valuation on $K(X)$ and hence is a stable limit of $\mathcal{W}_N.$ From the above argument it also follows that 
  \begin{align*}
	\rho_{\mathcal{W}_N}(f)=w(f), ~\forall~ f\in K[X]
\end{align*}
  and this proves  (a).
 
  Since $\vartheta_N\neq\emptyset,$ so by Lemma \ref{2.1.5} (i)   $Q_i\in\Phi(w_{Q_N},w)$ for every $i\in\vartheta_N,$ which implies that  $$\deg (\Phi(w_{N},w))=\deg (w_{N})$$ and $Q_N\notin\Phi(w_{N},w)$ because $w_{Q_N}(Q_N)=w(Q_N),$ proving  (b). Thus in view of (\ref{1.4}) and the claim,  after $N$ augmentation steps, we have that  $w$ is the stable limit of a continuous family  $\mathcal{W}_N=(\rho_i)_{i\in\vartheta_N},$ of augmentations of $w_N$:
 $$w_0\xrightarrow{Q_1,\gamma_1} w_1\xrightarrow{Q_2,\gamma_2}\cdots \longrightarrow w_{N-1}\xrightarrow{Q_N,\gamma_N} w_N\xrightarrow{(\rho_i)_{i\in\vartheta_N}} \rho_{\mathcal{W}_N}=w,$$  where $w_j=w_{Q_j},$ $\gamma_j=w(Q_j)$ for every $0\leq j\leq N,$ such that $\deg(\Phi(w_N,w) )=\deg(w_N)$ and $Q_N\notin\Phi(w_N,w).$
 
 Assume now that $I=\mathbb{N}\cup\{0\}.$  Then keeping in mind (1),  (2) and   (\ref{1.5}),  we have that
 $$w_0\xrightarrow{Q_1,\gamma_1} w_1\xrightarrow{Q_2,\gamma_2}\cdots \longrightarrow w_{j}\xrightarrow{Q_{j+1},\gamma_{j+1}} w_{j+1}\longrightarrow\cdots$$ is a complete infinite MLV chain of $w.$
Since $\{Q_i\}_{i\in\Delta}$ is an induced complete sequence of ABKPs for $w,$ so  $\mathcal{W}=(w_i=w_{Q_i})_{i\in\Delta}$ is a totally ordered family of valuations, taking values in a common value group, such that the bijection $i\mapsto w_{Q_i},$ is an isomorphism  between $\Delta$ and $\mathcal{W}.$ It only remains to prove that $w$ is the stable limit of $\mathcal{W}.$  For this it is enough to  show that every polynomial in $ K[X]$ is $\mathcal{W}$-stable. Let $f$ be any polynomial in $K[X].$ Since $\Lambda$ is an induced complete sequence of ABKPs for $w,$ so there exist some $j_0\in \Delta$ such that 
\begin{align*}
	w_{j_0}(f)=w(f).
\end{align*} 
For all $i>j_0$  in $\Delta,$ we have $\delta(Q_{j_0})<\delta(Q_i),$ which together with the above equality, on using   Proposition \ref{2.1.6} (ii) implies  that 
$	w_{i}(f)=w(f).$
Therefore,
$$w_{i}(f)=w_{j_0}(f),~\forall~ i>j_0\in\Delta,$$
so every polynomial  in $K[X]$ is $\mathcal{W}$-stable, i.e.,  $\rho_{\mathcal{W}}$ is a valuation on $K(X),$ and 
$$\rho_{\mathcal{W}}(f)=w_{Q_{j_0}}(f).$$ It also follows from the above argument that
$$\rho_{\mathcal{W}}(f)=w(f),~\forall ~f\in K[X].$$ 
Thus $w$ is the stable limit of $\mathcal{W}.$
		\end{proof}

	\begin{proof}[Proof of Theorem \ref{1.1.4}]
		Since $w$ has an MLV chain of type (\ref{1.1}) or (\ref{1.2}),
		 so
		 by Lemmas \ref{1.1.11},  \ref{2.1.1} and Remark \ref{2.1.9} we have that
		 \begin{enumerate}[(1)]
		 	\item if $w_j\longrightarrow w_{j+1}$ is an ordinary augmentation, i.e.,
		 	$w_{j+1}=[w_j;\phi_{j+1},\gamma_{j+1}]$ for some $\phi_{j+1}\in \operatorname{KP}(w_j),$
		 	 then $\phi_j$ and  $\phi_{j+1}$ are ABKPs for $w_{j+1}.$
		 	\item  if $w_j\longrightarrow w_{j+1}$ is a limit augmentation, i.e.,
		 	$w_{j+1}=[\mathcal{W}_j;\phi_{j+1},\gamma_{j+1}]$ for some  $\phi_{j+1}\in \operatorname{KP}_{\infty}(w_j),$ where $\mathcal{W}_j=(\rho_i=[w_i;\chi_i,\gamma_i])_{i\in\mathbf{A}_j}$ is an essential continuous family of augmentations of $w_j,$
		 	then $\phi_j,$ $\phi_{j+1}$ and $\chi_i$ for all $i\in\mathbf{A}_j$ are ABKPs for $w_{j+1}.$ 
		 \end{enumerate}
	  Now  by Remark \ref{1.1.20} and the  definition of an MLV chain of $w,$ if $w_j\longrightarrow w_{j+1}$ is an ordinary augmentation, then $$\phi_j\notin\Phi(w_j,w_{j+1})=[\phi_{j+1}]_{w_j},~\text{ i.e.,}~ \phi_{j+1}\nmid_{w_j}\phi_j,$$  and if $w_j\longrightarrow w_{j+1}$ is a  limit augmentation, then $$\phi_j\notin\Phi(w_j,w_{j+1})=[\chi_{i}]_{w_j},~  \text{i.e.,}~  \chi_{i}\nmid_{w_j}\phi_j,~\forall ~i\in\mathbf{A}_j.$$
	  Since, $w_j<w_{j+1}<w$ and $\rho_i<\rho_{i'}<w,$ so by Corollary \ref{1.1.16} (ii), we have that $$\Phi(w_j,w_{j+1})=\Phi(w_j,w)~\text{ and }~ \Phi(\rho_i,\rho_{i'})=\Phi(\rho_i,w),~ \text{for all $i<i'\in\mathbf{A}_j.$}$$   
	  Therefore, keeping in mind the definition of an MLV chain of $w$  together with  Theorem \ref{2.1.2}, 
	  Lemmas \ref{1.1.11},  \ref{2.1.1} and Remark \ref{2.1.9} it immediately follows that   
		 \begin{enumerate}[(a)]
		 	\item $\phi_j,$ $\phi_{j+1},$  $Q_i:=\chi_i,$ $i\in\mathbf{A}_j$ are ABKPs for $w,$ and 
		 	  $$ w_j=w_{\phi_j},~	 	\rho_{i}=w_{Q_{i}},~  w_{j+1}=w_{\phi_{j+1}}
	,~ \text{for all $i\in\mathbf{A}_j.$}$$
		 	\item $\deg \phi_j=\deg Q_i<\deg \phi_{j+1}$ for all $i\in\mathbf{A}_j.$
		 	\item  $\delta(\phi_j)<\delta(Q_i)<\delta(Q_{i'})<\delta(\phi_{j+1})$  for all $i<i'\in\mathbf{A}_j.$
		 	\item If $w_j\longrightarrow w_{j+1}$ is an ordinary augmentation, then $\phi_{j+1}\in\Phi(w_{\phi_j},w).$
		 \end{enumerate}
	 
	 Suppose first that $w$ is the stable limit of a continuous family of augmentations $\mathcal{W}_r=(\rho_i=[w_r;\chi_i,\gamma_i])_{i\in\mathbf{A}_r}$ of $w_r$
	 $$	w_0\xrightarrow{\phi_1,\gamma_1} w_1\xrightarrow{\phi_2,\gamma_2}\cdots \longrightarrow w_{r}\xrightarrow{(\rho_i)_{i\in\mathbf{A}_r}} \rho_{\mathcal{W}_r}=w.$$
	 Then from Lemma \ref{2.1.1}, we have that    $Q_i:=\chi_i,$ for all $i\in\mathbf{A}_r,$ is an ABKP for $w,$  
	 \begin{align}
	 	\rho_i&=w_{Q_i},~\text{ and} \label{1.15}\\
	 	\delta(\phi_r)&<\delta(Q_i)<\delta(Q_{i'})~\forall i<i'\in\mathbf{A}_r.\label{1.16}
	 \end{align}
		Let $I=\{0,1,\ldots, r\}$ and for every $j\in I,$ let $\Delta_j=\{j\}\cup\mathbf{A}_j,$ where in view of (1), (2)  $\mathbf{A}_j$ is either empty or an ordered set without a last element, for  $0\leq j\leq r-1.$   Let $\Delta=\bigcup_{j=0}^{r}\Delta_j,$   keeping in mind (1), (2) and the continuous family $\mathcal{W}_r,$   consider the set $\Lambda=\{Q_i\}_{i\in\Delta},$ where $Q_j:=\phi_j$ for every $j\in I.$  
		 Then in view of  (c) and (\ref{1.16}) we have that $$\delta(Q_i)<\delta(Q_{i'})~\forall i<i'\in\Delta.$$ Moreover,   the set $\Lambda$ satisfies the properties of Remark \ref{1.1.3}.    
		Let $f$ in $K[X]$ be any polynomial. Since  $w$ is the stable limit of $\mathcal{W}_r,$ so
		 $f$ is $\mathcal{W}_r$-stable.  Therefore, there exists some $i_0\in\mathbf{A}_r\subset\Delta$ such that 
		 \begin{align*}
		w(f)&=\rho_{i_0}(f)\\
		&=w_{Q_{i_0}}(f)~\text{(by (\ref{1.15}))}.
	\end{align*}
		If $\deg Q_r=\deg Q_{i_0}\leq\deg f,$ then  $\Lambda$  fulfills the condition to be a complete sequence of ABKPs for $w.$ Otherwise, $\deg Q_i\leq\deg f<\deg Q_{j+1}$ for  all $i\in\Delta_{j}$ and some minimal $j<r.$ Since  
		$\delta(c^{-1}f)=\delta(f),$ where $c\in K$ is the leading coefficient of $f,$ so we can assume  without loss of generality that $f$ is monic.
		If $\delta(f)\leq\delta(Q_i)$ for some $i\in\Delta_{j},$ then  clearly   $w_{Q_i}(f)=w(f).$ We now show that the case  $\delta(f)>\delta(Q_i)$ for every $i\in\Delta_{j},$ does not occur. If $\mathbf{A}_j=\emptyset,$  then  $f \in\Phi(w_{Q_j},w)$ and from (d) we also have that $Q_{j+1}\in\Phi(w_{Q_j},w),$ which implies that $\deg f=\deg Q_{j+1}.$  Therefore  $\mathbf{A}_j\neq\emptyset.$
		If  $\deg f=\deg Q_i,$   then    $f$ is  an ABKP for $w$ and by Remark \ref{1.1.20} and Corollary \ref{1.1.16},   $f$  belongs to the set $\{Q_{i}\mid i\in\mathbf{A}_j\},$   which  contradicts  the definition of $\mathbf{A}_j.$ 
		So $\deg Q_i<\deg f<\deg Q_{j+1}.$   Since for every $i\in\Delta_{j},$ $\delta(Q_i)<\delta(f),$  so  $w_{Q_i}(f)<w(f).$  
		In particular,  $w_{Q_i}(f)<w(f)$ for every $i\in\mathbf{A}_j,$   which in view of 
		 Proposition \ref{2.1.6} (ii), implies  that $w_{Q_{i}}(f)<w_{Q_{i'}} (f)$ for every $i<i '\in\mathbf{A}_j.$ By definition of MLV chain, it now follows that $\deg f\geq \deg Q_{j+1},$ which is not the case.
		Hence $\Lambda$ is an induced complete sequence of ABKPs for $w$ such that $\Delta$ has no last element.

		Assume now that $w$ is the stable limit of a complete infinite MLV chain,
		$$	w_0\xrightarrow{\phi_1,\gamma_1} w_1\xrightarrow{\phi_2,\gamma_2}\cdots \longrightarrow w_{j}\xrightarrow{\phi_{j+1},\gamma_{j+1}} w_{j+1}\longrightarrow\cdots$$
		i.e., $w$ is the stable limit of $\mathcal{W}=(w_i)_{i\in I},$ where $I=\mathbb{N}\cup\{0\}$ (see Remark \ref{1.1.10}).
		  For every $j\in I,$ let $\Delta_j=\{j\}\cup\mathbf{A}_j,$ where $\mathbf{A}_j$ is  an ordered set without a last element, whenever $w_j\longrightarrow w_{j+1}$ is a limit augmentation,  or  is an empty set, if $w_j\longrightarrow w_{j+1}$ is an ordinary augmentation. Let   $\Delta=\bigcup_{j\in I}\Delta_j,$
		   and, we denote $Q_j:=\phi_j$ for every $j\in I.$  
		    Keeping in mind (1), (2)  consider the set $\Lambda=\{Q_i\}_{i\in\Delta}.$ For any $i<i'\in\Delta$ by (c) we have that $$\delta(Q_i)<\delta(Q_{i'}),$$ and therefore the set $\Lambda$ satisfies the properties of Remark \ref{1.1.3}. 
		   Now for any
		  polynomial $f\in K[X],$ on using the fact that $w$ is the stable limit of $\mathcal{W},$ we get   $$w(f)=w_{Q_{i_0}}(f) ~\text{for some}~ i_0\in\Delta.$$
		   Now, arguing  similarly as in the previous case, we have that for every polynomial $f\in K[X]$ there exists a polynomial $Q_i\in\Lambda$ such that $\deg Q_i \leq\deg f$ and $w_{Q_i}(f)=w(f).$  Hence,  $\{Q_i\}_{i\in\Delta}$ is an induced complete sequence of ABKPs for $w$ such that $\Delta$ has no last element.
		   
		    Thus, in either case  $\Lambda=\{Q_i\}_{i\in\Delta}$ is  an induced complete sequence of ABKPs for $w$ such that   
		   \begin{itemize}
		   	\item $\Delta=\bigcup_{j\in I}\Delta_j,$ with $I=\{0,1,\ldots, r\}$ or $\mathbb{N}\cup\{0\}$ and  $\Delta_j=\{j\}\cup\mathbf{A}_j$ for all $j\in I.$ Moreover,  $Q_j=\phi_j$ for all $j\in I.$
		   	\item  $w_j\longrightarrow w_{j+1}$ is an ordinary augmentation if and only if  $\vartheta_j=\emptyset.$ 
		   	\item if $w_j\longrightarrow w_{j+1}$ is a limit augmentation (or $w_j\longrightarrow w$ is a stable limit step) 
		   	with respect to an essential continuous family (or continuous family) $\mathcal{W}_j=(\rho_i)_{i\in\mathbf{A}_j}$ of augmentations of $w_j,$   then $\vartheta_{j}=\mathbf{A}_j$ and   $Q_i=\chi_i$ for all $i\in\mathbf{A}_j.$ 
		   \end{itemize}
	\end{proof}
	
	\section*{Acknowledgement}
	We would like to thank the anonymous referee for a careful reading and providing useful suggestions which led to an improvement in the presentation of the paper.
	The	research of  first author is supported by CSIR (Grant No.\  09/045(1747)/2019-EMR-I).

	\bibliographystyle{amsplain}
	
\end{document}